\documentclass[reqno, 12pt]{amsart}
 \usepackage[a4paper, hmargin = 1.25in, vmargin = 1.5in]{geometry}
 \usepackage{amsmath, amssymb, amsfonts}
 \usepackage{amsthm}
 
 \newtheorem{thm}{Theorem}[section]

 \theoremstyle{definition}
 \newtheorem{defn}{Definition}[section]
 \theoremstyle{remark}
 
 \numberwithin{equation}{section}

\begin{document}

\begin{center}

\begin{title}
\title{\bf\Large{{Lyapunov-Type Inequalities for Discrete Riemann-Liouville Fractional Boundary Value Problems}}}
\end{title}

\vskip 0.25 in

\begin{author}
\author {Jagan Mohan Jonnalagadda\footnote[1]{Department of Mathematics, Birla Institute of Technology and Science Pilani, Hyderabad - 500078, Telangana, India. email: j.jaganmohan@hotmail.com}}
\end{author}

\end{center}

\vskip 0.25 in

\noindent{\bf Abstract:} In this article we establish a few Lyapunov-type inequalities for two-point discrete fractional boundary value problems involving Riemann-Liouville type backward differences. To illustrate the applicability of established results, we obtain criteria for the nonexistence of nontrivial solutions and estimate lower bounds for eigenvalues of the corresponding eigenvalue problems. We also apply these inequalities to deduce criteria for the nonexistence of real zeros of certain discrete Mittag-Leffler functions.

\vskip 0.25 in

\noindent{\bf Key Words:} Fractional order, backward (nabla) difference, boundary value problem, Green's function, Lyapunov inequality, discrete Mittag-Leffler function

\vskip 0.25 in

\noindent{\bf AMS Classification:} Primary 34A08, 34B05, 26D15; Secondary 39A10, 39A12.

\vskip 0.25 in

\section{Introduction}

In 1907, Lyapunov \cite{L} provided a necessary condition, known as the Lyapunov inequality, for the existence of a nontrivial solution of Hill's equation associated with Dirichlet boundary conditions.
\begin{thm} \cite{L} \label{ODE}
If the boundary value problem
\begin{equation} \label{BVP}
\begin{cases}
y''(t) + p(t)y(t) = 0, \quad a < t < b,\\
y(a) = 0, \quad y(b) = 0, 
\end{cases}
\end{equation}
has a nontrivial solution, where $p : [a, b] \rightarrow \mathbb{R}$ is a continuous function, then
\begin{equation} \label{Lyp}
\int^{b}_{a}|p(s)|ds > \frac{4}{(b - a)}.
\end{equation}
\end{thm}
The Lyapunov inequality \eqref{Lyp} is applicable in various problems related to the theory differential equations. Due to its importance, the Lyapunov inequality has been generalized in many forms. For more details, we refer \cite{B,PA,PI,AT,Y1,Y2} and the references therein.

In this line, several authors \cite{CT,D,F1,F2,J1,J2,O,R,S} have investigated Lyapunov-type inequalities for various classes of fractional boundary value problems. In 2013, Ferreira \cite{F1} generalized Theorem \ref{ODE} to the case where the classical second-order derivative in \eqref{BVP} is replaced by an $\alpha^{\text{th}}$-order ($1 < \alpha \leq 2$) Riemann-Liouville derivative.
\begin{thm} \cite{F1} \label{FDE RL}
If the fractional boundary value problem
\begin{equation*} \label{BVP RL}
\nonumber \begin{cases}
\big{(}D^{\alpha}_{a}y\big{)}(t) + p(t)y(t) = 0, \quad a < t < b,\\
y(a) = 0, \quad y(b) = 0, 
\end{cases}
\end{equation*}
has a nontrivial solution, where $p : [a, b] \rightarrow \mathbb{R}$ is a continuous function, then
\begin{equation*} \label{Lyp RL}
\int^{b}_{a}|p(s)|ds > \Gamma(\alpha) \Big{(}\frac{4}{b - a}\Big{)}^{\alpha - 1}.
\end{equation*}
\end{thm}

Moving to discrete calculus, in 1983, Cheng \cite{C1} developed the following discrete analogue of the Lyapunov inequality for the first time.
\begin{thm} \cite{C1} \label{D ODE}
If the boundary value problem
\begin{equation*} \label{D BVP}
\begin{cases}
\big{(}\Delta^{2}u\big{)}(t - 1) + q(t)y(t) = 0, \quad t \in \mathbb{N}^b_{a},\\
u(a - 1) = 0, \quad u(b + 1) = 0, 
\end{cases}
\end{equation*}
has a nontrivial solution, where $q$ is a nonnegative function defined on $\mathbb{N}^b_{a}$, then
\begin{equation*} \label{D LyP}
\sum^{b}_{s = a}q(s) > \begin{cases}
\frac{2(b - a + 1) + 1}{(b - a + 1)(b - a + 2)}, \quad \mathrm{if} ~ (b - a) ~ \mathrm{is ~ odd},\\
\frac{2}{(b - a + 2)},  \quad \hspace{32pt} \mathrm{if} ~ (b - a) ~ \mathrm{is ~ even}.
\end{cases}
\end{equation*}
\end{thm}
Here $\Delta^{2}$ denotes the classical second-order difference operator. For a detailed discussion, we refer the reader to \cite{C2,CL,GU,H,LY,U,ZT} and the references therein.

On the other hand, in 2015, Ferreira \cite{F3} generalized Theorem \ref{D ODE} for $\alpha^{\text{th}}$-order ($1 < \alpha \leq 2$) Riemann-Liouville type forward differences ${\Delta}^{\alpha}$.
\begin{thm} \cite{F3} \label{D FDE RL}
If the fractional boundary value problem
\begin{equation*} \label{D BVP RL}
\begin{cases}
\big{(}{\Delta}^{\alpha}u\big{)}(t) + q(t + \alpha - 1)u(t + \alpha - 1) = 0, \quad t \in \mathbb{N}^{b + 1}_{1}, \quad b \geq 2, \\
y(\alpha - 2) = 0, \quad y(\alpha + b + 1) = 0, 
\end{cases}
\end{equation*}
has a nontrivial solution, where $q$ is a nonnegative function defined on $\{\alpha, \alpha + (\alpha - 1), \alpha + 2(\alpha - 1), \cdots, \alpha + b\}$, then
\begin{equation*} \label{D Lyp RL 1}
\sum^{b + 1}_{s = 0}q(s + \alpha - 1) > \begin{cases}
4 \Gamma(\alpha)\frac{\Gamma(b + \alpha + 2)\Gamma^{2}(\frac{b}{2} + 2)}{(b + 2\alpha)(b + 2)\Gamma^{2}(\frac{b}{2} + \alpha)\Gamma(b + 3)}, \quad \mathrm{if} ~ b ~ \mathrm{is ~ even},\\
\Gamma(\alpha)\frac{\Gamma(b + \alpha + 2)\Gamma^{2}(\frac{b + 3}{2} + 2)}{\Gamma(b + 3)\Gamma^{2}(\frac{b + 1}{2} + \alpha)}, \quad \hspace{34pt} \mathrm{if} ~ b ~ \mathrm{is ~ odd}.
\end{cases}
\end{equation*}
\end{thm}
Following Ferreira, authors of \cite{CT 1,CU,GG} have obtained generalized Lyapunov-type inequalities for various classes of delta fractional boundary value problems.

Motivated by these developments, in this article, we derive Lyapunov-type inequalities for two-point nabla Riemann-Liouville fractional boundary value problems.

\section{Preliminaries}
Throughout, we shall use the following notations, definitions and known results of fractional nabla calculus \cite{T,G,FP1,CH}. Denote the set of all real numbers by $\mathbb{R}$. Define $$\mathbb{N}_{a} : = \{a, a + 1, a + 2, \cdots\} ~ \text{and} ~ \mathbb{N}^{b}_{a} : = \{a, a + 1, a + 2, \cdots, b\},$$ for any $a$, $b \in \mathbb{R}$ such that $(b - a)$ is a positive integer. Assume that empty sums and products are taken to be 0 and 1, respectively.

\begin{defn} (Backward Jump Operator) \cite{BP}
The backward jump operator $\rho : \mathbb{N}_{a} \rightarrow \mathbb{N}_{a}$ is defined by $$\rho(t) = \max\{a, (t - 1)\}, ~ t \in \mathbb{N}_{a}.$$
\end{defn}

\begin{defn} (Gamma Function) \cite{KIL,P1} The Euler gamma function is defined by $$\Gamma (z) : = \int_0^\infty t^{z - 1} e^{-t} dt, \quad \Re(z) > 0.$$ Using the reduction formula $$ \Gamma (z + 1) = z\Gamma (z), \quad \Re(z) > 0,$$ the Euler gamma function can be extended to the half-plane $\Re(z) \leq 0$ except for $z \neq 0, -1, -2, \cdots$.
\end{defn}

\begin{defn} (Generalized Rising Function) \cite{CH}
For $t \in \mathbb{R} \setminus \{\cdots, -2, -1, 0\}$ and $r \in \mathbb{R}$ such that $(t + r) \in \mathbb{R} \setminus \{\cdots, -2, -1, 0\}$, the generalized rising function is defined by
\begin{equation}
\nonumber t^{\overline{r}} = \frac{\Gamma(t + r)}{\Gamma(t)}, \quad 0^{\overline{r}} : = 0.
\end{equation}
\end{defn}

\begin{defn} (Nabla Difference) \cite{BP}
Let $u: \mathbb{N}_{a} \rightarrow \mathbb{R}$ and $N \in \mathbb{N}_1$. The first order backward (nabla) difference of $u$ is defined by $$\big{(}\nabla u\big{)}(t) : = u(t) - u(t - 1), \quad t \in \mathbb{N}_{a + 1},$$ and the $N^{th}$-order nabla difference of $u$ is defined recursively by $$\big{(}{\nabla}^{N}u\big{)}(t) : = \Big{(}\nabla\big{(}\nabla^{N - 1}u\big{)}\Big{)}(t), \quad t\in \mathbb{N}_{a + N}.$$
\end{defn}

\begin{defn} (Nabla Sum) \cite{CH}
Let $u: \mathbb{N}_{a + 1} \rightarrow \mathbb{R}$ and $N \in \mathbb{N}_1$. The $N^{\text{th}}$-order nabla sum of $u$ based at $a$ is given by
\begin{equation}
\nonumber \big{(}\nabla ^{-N}_{a}u\big{)}(t) : = \frac{1}{(N - 1)!}\sum^{t}_{s = a + 1}(t - \rho(s))^{\overline{N - 1}}u(s),
~ t \in \mathbb{N}_{a + 1}.
\end{equation}
We define $\big{(}\nabla ^{-0}_{a}u\big{)}(t) = u(t)$, $t \in \mathbb{N}_{a + 1}$.
\end{defn}

\begin{defn} (Nabla Fractional Sum) \cite{CH}
Let $u: \mathbb{N}_{a + 1} \rightarrow \mathbb{R}$ and $\nu > 0$. The $\nu^{\text{th}}$-order nabla sum of $u$ based at $a$ is given by
\begin{equation}
\nonumber \big{(}\nabla ^{-\nu}_{a}u\big{)}(t) = \frac{1}{\Gamma(\nu)}\sum^{t}_{s = a + 1}(t - \rho(s))^{\overline{\nu - 1}}u(s),
~ t \in \mathbb{N}_{a + 1}.
\end{equation} 
\end{defn}

\begin{defn} (Nabla Fractional Difference) \cite{CH}
Let $u: \mathbb{N}_{a + 1} \rightarrow \mathbb{R}$, $\nu > 0$ and choose $N \in \mathbb{N}_1$ such that $N - 1 < \nu \leq N$. The Riemann-Liouville type $\nu^{\text{th}}$-order nabla difference of $u$ is given by
\begin{equation}
\nonumber \big{(}\nabla ^{\nu}_{a}u\big{)}(t) = \Big{(}\nabla^N\big{(}\nabla_{a}^{-(N - \nu)}u\big{)}\Big{)}(t), ~ t\in\mathbb{N}_{a + N}.
\end{equation}
\end{defn}

\begin{thm} (Alternative Definition of the Nabla Fractional Difference) \cite{AH}
Let $u: \mathbb{N}_{a} \rightarrow \mathbb{R}$, $\nu > 0$, $\nu \not\in \mathbb{N}_{1}$, and choose $N \in \mathbb{N}_1$ such that $N - 1 < \nu < N$. Then,
$$\big{(}\nabla ^{\nu}_{a}u\big{)}(t) = \frac{1}{\Gamma(-\nu)}\sum^{t}_{s = a + 1}(t - \rho(s))^{\overline{-\nu - 1}}u(s), \quad t \in \mathbb{N}_{a + 1}.$$
\end{thm}

We observe the following properties of gamma and generalized rising functions.
\begin{thm} \cite{J} \label{Gamma}
Assume the following gamma and generalized rising functions are well defined.
\begin{enumerate}
\item $\Gamma(t) > 0$ for all $t > 0$;
\item $t^{\overline\alpha}(t + \alpha)^{\overline \beta} = t^{\overline {\alpha + \beta}}$;
\item If $t \leq r$  then $t^{\overline{\alpha}} \leq r^{\overline{\alpha}}$;
\item If $\alpha < t \leq r$ then $r^{\overline{-\alpha}} \leq t^{\overline{-\alpha}}$;
\item $\nabla(t + \alpha)^{\overline {\beta}} = \beta(t + \alpha)^{\overline {\beta - 1}}$;
\item $\nabla(\alpha - t)^{\overline {\beta}} = -\beta(\alpha - \rho(t))^{\overline {\beta - 1}}$.
\end{enumerate}
\end{thm}

\begin{thm} (Generalized Power Rules) \cite{TF} \label{Power Rule} Let $\nu > 0$ and $\mu \in \mathbb{R}$ such that $\mu$, $\mu + \nu$ and $\mu - \nu$ are nonnegative integers. Then,
\begin{align*}
& \nabla_{a}^{-\nu}(t - a)^{\overline{\mu}} = \frac{\Gamma(\mu + 1)}{\Gamma(\mu + \nu + 1)}(t - a)^{\overline{\mu + \nu}}, \quad t \in \mathbb{N}_{a + 1},\\
& \nabla_{a}^{\nu}(t - a)^{\overline{\mu}} = \frac{\Gamma(\mu + 1)}{\Gamma(\mu - \nu + 1)}(t - a)^{\overline{\mu - \nu}}, \quad t \in \mathbb{N}_{a + 1}.
\end{align*}
\end{thm}

\begin{defn} (Nabla Mittag-Leffler Function) \cite{CH} \label{DML}
For $|p| < 1$, $\alpha > 0$ and $\beta \in \mathbb{R}$, we define the nabla Mittag-Leffler function by $$E_{p, \alpha, \beta}(t, a) = \sum^{\infty}_{k = 0}p^{k}\frac{(t - a)^{\overline{\alpha k + \beta}}}{\Gamma(\alpha k + \beta + 1)}, \quad t \in \mathbb{N}_{a}.$$ Clearly, we have $E_{0, \alpha, 0}(t, a) = 1$ and $E_{p, \alpha, 0}(a, a) = 1$.
\end{defn}

\begin{thm} \cite{CH} \label{FD DML}
Assume $|p| < 1$, $\alpha > 0$ and $\beta \in \mathbb{R}$. Then, $$\nabla^{\nu}_{a}E_{p, \alpha, \beta}(t, a) = E_{p, \alpha, \beta - \nu}(t, a), \quad t \in \mathbb{N}_{a + 1}.$$
\end{thm}

\begin{thm} \cite{CH} \label{Difference}
Let $\nu$, $\mu > 0$ and $u : \mathbb{N}_{a} \rightarrow \mathbb{R}$. Then, $$\Big{(}\nabla^{\nu}_{a}\big{(}\nabla^{-\mu}_{a}u\big{)}\Big{)}(t) = \big{(}\nabla^{\nu - \mu}_{a}u\big{)}(t), \quad t \in \mathbb{N}_{a + 1}.$$
\end{thm}

\begin{thm} \cite{CH} \label{Sum}
Assume $\nu > 0$ and $N - 1 < \nu \leq N$. Then, a general solution of $$\big{(}\nabla^{\nu}_{a}u\big{)}(t) = 0, \quad t \in \mathbb{N}_{a + N},$$ is given by $$u(t) = C_{1}(t - a)^{\overline{\nu - 1}} + C_{2}(t - a)^{\overline{\nu - 2}} + \cdots + C_{N}(t - a)^{\overline{\nu - N}}, \quad t \in \mathbb{N}_{a + 1},$$ where $C_{1}, C_{2}, \cdots, C_{N} \in \mathbb{R}$.
\end{thm}

\begin{thm} \cite{CH} \label{SUM ML}
Assume $\nu > 0$, $N - 1 < \nu \leq N$ and $|c| < 1$. Then, a general solution of $$\big{(}\nabla^{\nu}_{a}u\big{)}(t) + c u(t) = 0, \quad t \in \mathbb{N}_{a + N},$$ is given by $$u(t) = C_{1}E_{-c, \nu, \nu - 1}(t, a) + C_{2}E_{-c, \nu, \nu - 2}(t, a) + \cdots + C_{N}E_{-c, \nu, \nu - N}(t, a), \quad t \in \mathbb{N}_{a + 1},$$ where $C_{1}, C_{2}, \cdots, C_{N} \in \mathbb{R}$.
\end{thm}

\section{Left-Focal Type Boundary Value Problem}

In this section, we derive a few important properties of the Green's function for a left-focal type discrete boundary value problem and obtain the corresponding Lyapunov-type inequality.

\begin{thm}
Let $1 < \alpha < 2$ and $h : \mathbb{N}^{b}_{a + 1} \rightarrow \mathbb{R}$. The discrete fractional boundary value problem
\begin{equation} \label{LF Theorem 1}
\begin{cases}
\big{(}\nabla^{\alpha}_{a}u\big{)}(t) + h(t) = 0, \quad t \in \mathbb{N}^{b}_{a + 2}, \\ \big{(}\nabla^{\alpha - 1}_{a} u\big{)}(a + 1) = 0, ~ u(b) = 0,
\end{cases}
\end{equation}
has the unique solution 
\begin{equation} \label{LF Solution}
u(t) = \sum^{b}_{s = a + 2}G_{l}(t, s)h(s), \quad t \in \mathbb{N}^{b}_{a + 1},
\end{equation}
where
\begin{equation} \label{LF Green}
G_{l}(t, s) = \begin{cases}
\frac{1}{\Gamma(\alpha)}\frac{(b - s + 1)^{\overline{\alpha - 1}}(t - a)^{\overline{\alpha - 2}}}{(b - a)^{\overline{\alpha - 2}}}, \hspace{108pt} t \in \mathbb{N}^{\rho(s)}_{a + 1},\\
\frac{1}{\Gamma(\alpha)}\big{[}\frac{(b - s + 1)^{\overline{\alpha - 1}}(t - a)^{\overline{\alpha - 2}}}{(b - a)^{\overline{\alpha - 2}}} - (t - s + 1)^{\overline{\alpha - 1}}\big{]}, \quad t \in \mathbb{N}^{b}_{s}.
\end{cases}
\end{equation}
\end{thm}

\begin{proof}
Applying $\nabla^{-\alpha}_{a}$ on both sides of \eqref{LF Theorem 1} and using Theorem \ref{Sum}, we have
\begin{equation} \label{Sol 1}
u(t) = -\big{(}\nabla^{-\alpha}_{a}h\big{)}(t) + C_{1}(t - a)^{\overline{\alpha - 1}} + C_{2}(t - a)^{\overline{\alpha - 2}}, \quad t \in \mathbb{N}_{a + 1},
\end{equation}
for some $C_{1}$, $C_{2} \in \mathbb{R}$. Applying $\nabla^{\alpha - 1}_{a}$ on both sides of \eqref{Sol 1} and using Theorems \ref{Power Rule} - \ref{Difference}, we have
\begin{equation} \label{Sol 2}
\big{(}\nabla^{\alpha - 1}_{a} u\big{)}(t) = -\big{(}\nabla^{-1}_{a}h\big{)}(t) + C_{1}\Gamma(\alpha), \quad t \in \mathbb{N}_{a + 1}.
\end{equation} 
Using $\big{(}\nabla^{\alpha - 1}_{a} u\big{)}(a + 1) = 0$ in \eqref{Sol 2}, we get $$C_{1} = \frac{h(a + 1)}{\Gamma(\alpha)}.$$ Using $u(b) = 0$ in \eqref{Sol 1}, we get $$C_{2} = \frac{1}{(b - a)^{\overline{\alpha - 2}}\Gamma(\alpha)}\sum^{b}_{s = a + 2}(b - s + 1)^{\overline{\alpha - 1}}h(s).$$ Substituting the values of $C_{1}$ and $C_{2}$ in \eqref{Sol 1}, we have
\begin{align*}
u(t) & = -\frac{1}{\Gamma(\alpha)}\sum^{t}_{s = a + 1}(t - s + 1)^{\overline{\alpha - 1}}h(s) + \frac{h(a + 1)}{\Gamma(\alpha)} (t - a)^{\overline{\alpha - 1}} \\ & + \frac{(t - a)^{\overline{\alpha - 2}}}{(b - a)^{\overline{\alpha - 2}}\Gamma(\alpha)}\sum^{b}_{s = a + 2}(b - s + 1)^{\overline{\alpha - 1}}h(s) \\ & = -\frac{1}{\Gamma(\alpha)}\sum^{t}_{s = a + 2}(t - s + 1)^{\overline{\alpha - 1}}h(s) + \frac{(t - a)^{\overline{\alpha - 2}}}{(b - a)^{\overline{\alpha - 2}}\Gamma(\alpha)}\sum^{b}_{s = a + 2}(b - s + 1)^{\overline{\alpha - 1}}h(s) \\ & = \frac{1}{\Gamma(\alpha)}\sum^{t}_{s = a + 2}\Big{[}\frac{(b - s + 1)^{\overline{\alpha - 1}}(t - a)^{\overline{\alpha - 2}}}{(b - a)^{\overline{\alpha - 2}}} - (t - s + 1)^{\overline{\alpha - 1}}\Big{]}h(s) \\ & + \frac{1}{\Gamma(\alpha)}\sum^{b}_{s = t + 1}\frac{(b - s + 1)^{\overline{\alpha - 1}}(t - a)^{\overline{\alpha - 2}}}{(b - a)^{\overline{\alpha - 2}}}h(s) \\ & = \sum^{b}_{s = a + 2}G_{l}(t, s)h(s).
\end{align*}
\end{proof}

First, we show that this Green's function is nonnegative and obtain an upper bound for the Green's function and its integral.

\begin{thm} \label{LF Green Positive}
The Green's function $G_{l}(t, s)$ satisfies $G_{l}(t, s) \geq 0$ for $(t, s) \in \mathbb{N}^{b}_{a + 1} \times \mathbb{N}^{b}_{a + 2}$.
\end{thm}

\begin{proof}
For $t \in \mathbb{N}^{\rho(s)}_{a + 1}$, using Theorem \ref{Gamma}, 
\begin{align*}
G_{l}(t, s) & = \frac{(b - s + 1)^{\overline{\alpha - 1}}(t - a)^{\overline{\alpha - 2}}}{\Gamma(\alpha)(b - a)^{\overline{\alpha - 2}}} \\ & = \frac{\Gamma(b - s + \alpha)\Gamma(t - a + \alpha - 2)\Gamma(b - a)}{\Gamma(\alpha)\Gamma(b - s + 1)\Gamma(b - a + \alpha - 2)\Gamma(t - a)} > 0. 
\end{align*}
Now, suppose $t \in \mathbb{N}^{b}_{s}$. Since $t \leq b$ and $(\alpha - 2) < (t - a) \leq (b - a)$, by Theorem \ref{Gamma}, we have $$(t - s + 1)^{\overline{\alpha - 1}} \leq (b - s + 1)^{\overline{\alpha - 1}} ~ \text{and} ~ (b - a)^{\overline{\alpha - 2}} \leq (t - a)^{\overline{\alpha - 2}},$$ implies 
\begin{align*}
G_{l}(t, s) & = \frac{1}{\Gamma(\alpha)}\Big{[}\frac{(b - s + 1)^{\overline{\alpha - 1}}(t - a)^{\overline{\alpha - 2}}}{(b - a)^{\overline{\alpha - 2}}} - (t - s + 1)^{\overline{\alpha - 1}}\Big{]} \\ & \geq \frac{(b - s + 1)^{\overline{\alpha - 1}}}{\Gamma(\alpha)}\Big{[}\frac{(t - a)^{\overline{\alpha - 2}}}{(b - a)^{\overline{\alpha - 2}}} - 1\Big{]} \geq 0.
\end{align*}
Hence the proof.
\end{proof}

\begin{thm} \label{LF Green Max}
The maximum of the Green's function $G_{l}(t, s)$ defined in \eqref{LF Green} is given by $$\max_{(t, s) \in \mathbb{N}^{b}_{a + 1} \times \mathbb{N}^{b}_{a + 2}}G_{l}(t, s) = \frac{(b - a - 1)}{(\alpha - 1)}.$$
\end{thm}

\begin{proof}
Fix $s \in \mathbb{N}^{b}_{a + 2}$. Let $t \in \mathbb{N}^{\rho(s)}_{a + 2}$. Consider 
\begin{align*}
\nabla_{t}\big{[}G_{l}(t, s)\big{]} & = \frac{1}{\Gamma(\alpha)}\frac{(b - s + 1)^{\overline{\alpha - 1}}}{(b - a)^{\overline{\alpha - 2}}}\nabla_{t}\big{[}(t - a)^{\overline{\alpha - 2}}\big{]} \\ & = -\frac{(2 - \alpha)}{\Gamma(\alpha)}\frac{(b - s + 1)^{\overline{\alpha - 1}}(t - a)^{\overline{\alpha - 3}}}{(b - a)^{\overline{\alpha - 2}}} \\ & = -\frac{(2 - \alpha)}{\Gamma(\alpha)}\frac{\Gamma(b - s + \alpha)\Gamma(t - a + \alpha - 3)\Gamma(b - a)}{\Gamma(b - s + 1)\Gamma(b - a + \alpha - 2)\Gamma(t - a)} < 0,
\end{align*}
implies $G_{l}(t, s)$ is a decreasing function of $t$. Now, suppose $t \in \mathbb{N}^{b}_{s}$. Consider 
\begin{align}
\nonumber \nabla_{t}\big{[}G_{l}(t, s)\big{]} & = \frac{1}{\Gamma(\alpha)}\Big{[}\frac{(b - s + 1)^{\overline{\alpha - 1}}}{(b - a)^{\overline{\alpha - 2}}}\nabla_{t}\big{[}(t - a)^{\overline{\alpha - 2}}\big{]} - \nabla_{t}\big{[}(t - s + 1)^{\overline{\alpha - 1}}\big{]}\Big{]} \\ \nonumber & = -\frac{(2 - \alpha)}{\Gamma(\alpha)}\frac{(b - s + 1)^{\overline{\alpha - 1}}(t - a)^{\overline{\alpha - 3}}}{(b - a)^{\overline{\alpha - 2}}} - \frac{(\alpha - 1)}{\Gamma(\alpha)}(t - s + 1)^{\overline{\alpha - 2}} \\ \nonumber & = -\frac{(2 - \alpha)}{\Gamma(\alpha)}\frac{\Gamma(b - s + \alpha)\Gamma(t - a + \alpha - 3)\Gamma(b - a)}{\Gamma(b - s + 1)\Gamma(b - a + \alpha - 2)\Gamma(t - a)} \\ \nonumber & - \frac{(\alpha - 1)}{\Gamma(\alpha)}\frac{\Gamma(t - s + \alpha - 1)}{\Gamma(t - s + 1)} < 0,
\end{align}
implies $G_{l}(t, s)$ is a decreasing function of $t$. Now, we examine the Green's function to determine whether the maximum for a fixed $s$ will occur at $(a + 1, s)$, $(a + 2, s)$ or $(s, s)$. We have
\begin{equation*}
G_{l}(a + 1, s) = \frac{(b - s + 1)^{\overline{\alpha - 1}}}{(\alpha - 1)(b - a)^{\overline{\alpha - 2}}},
\end{equation*}
\begin{equation*}
G_{l}(a + 2, s) = \frac{(b - s + 1)^{\overline{\alpha - 1}}}{(b - a)^{\overline{\alpha - 2}}},
\end{equation*}
and
\begin{equation*}
G_{l}(s, s) = \frac{(b - s + 1)^{\overline{\alpha - 1}}(s - a)^{\overline{\alpha - 2}}}{\Gamma(\alpha)(b - a)^{\overline{\alpha - 2}}} - 1.
\end{equation*}
Clearly, $G_{l}(a + 2, s) < G_{l}(a + 1, s)$ and $$\max_{s \in \mathbb{N}^{b}_{a + 2}}G_{l}(a + 1, s) =  \frac{(b - a - 1)^{\overline{\alpha - 1}}}{(\alpha - 1)(b - a)^{\overline{\alpha - 2}}} = \frac{(b - a - 1)}{(\alpha - 1)}.$$ Also, $$\max_{s \in \mathbb{N}^{b}_{a + 2}}G_{l}(s, s) =  \frac{(b - a - 1)^{\overline{\alpha - 1}}}{(b - a)^{\overline{\alpha - 2}}} - 1 = (b - a - 2).$$ Thus, $$\max_{(t, s) \in \mathbb{N}^{b}_{a + 1} \times \mathbb{N}^{b}_{a + 2}}G_{l}(t, s) = \frac{(b - a - 1)}{(\alpha - 1)}.$$
\end{proof}

\begin{thm} \label{LF Green Int Max}
The following inequality holds for the Green's function $G_{l}(t, s)$ from \eqref{LF Green}. $$\sum^{b}_{s = a + 2}G_{l}(t, s) \leq \frac{(b - a - 1)(b - a + \alpha - 2)}{\alpha(\alpha - 1)}, \quad (t, s) \in \mathbb{N}^{b}_{a + 1} \times \mathbb{N}^{b}_{a + 2}.$$
\end{thm}

\begin{proof}
Consider 
\begin{align*}
\sum^{b}_{s = a + 2}G_{l}(t, s) & = \frac{1}{\Gamma(\alpha)}\sum^{t}_{s = a + 2}\Big{[}\frac{(b - s + 1)^{\overline{\alpha - 1}}(t - a)^{\overline{\alpha - 2}}}{(b - a)^{\overline{\alpha - 2}}} - (t - s + 1)^{\overline{\alpha - 1}}\Big{]} \\ & + \frac{1}{\Gamma(\alpha)}\sum^{b}_{s = t + 1}\frac{(b - s + 1)^{\overline{\alpha - 1}}(t - a)^{\overline{\alpha - 2}}}{(b - a)^{\overline{\alpha - 2}}}\\ & = \frac{(t - a)^{\overline{\alpha - 2}}}{(b - a)^{\overline{\alpha - 2}}}\sum^{b}_{s = a + 2}\frac{(b - s + 1)^{\overline{\alpha - 1}}}{\Gamma(\alpha)} - \sum^{t}_{s = a + 2}\frac{(t - s + 1)^{\overline{\alpha - 1}}}{\Gamma(\alpha)} \\ & = \frac{(t - a)^{\overline{\alpha - 2}}}{(b - a)^{\overline{\alpha - 2}}}\frac{(b - a - 1)^{\overline{\alpha}}}{\Gamma(\alpha + 1)} - \frac{(t - a - 1)^{\overline{\alpha}}}{\Gamma(\alpha + 1)}.
\end{align*}
We now find the maximum of this expression with respect to $t \in \mathbb{N}^{b}_{a + 1}$. Using Theorem \ref{Gamma}, for $t \in \mathbb{N}^{b}_{a + 1}$, $$\frac{(t - a - 1)^{\overline{\alpha}}}{\Gamma(\alpha + 1)} = \frac{\Gamma(t - a + \alpha - 1)}{\Gamma(\alpha + 1)\Gamma(t - a - 1)} \geq 0.$$ Thus, $$\max_{t \in \mathbb{N}^{b}_{a + 1}}\Big{[}\sum^{b}_{s = a + 2}G_{l}(t, s)\Big{]} = \frac{(b - a - 1)^{\overline{\alpha}}}{\alpha(\alpha - 1)(b - a)^{\overline{\alpha - 2}}} = \frac{(b - a - 1)(b - a + \alpha - 2)}{\alpha(\alpha - 1)}.$$ Hence the proof.
\end{proof}

We are now able to formulate a Lyapunov-type inequality for the left-focal type discrete boundary value problem.

\begin{thm} \label{LF Theorem 2}
If the following discrete fractional boundary value problem
\begin{equation} \label{LF FDE 2}
\begin{cases}
\big{(}\nabla^{\alpha}_{a}u\big{)}(t) + q(t)y(t) = 0, \quad t \in \mathbb{N}^{b}_{a + 2}, \\ \big{(}\nabla^{\alpha - 1}_{a} u\big{)}(a + 1) = 0, ~ u(b) = 0,
\end{cases}
\end{equation}
has a nontrivial solution, then
\begin{equation} \label{LF Lyp}
\sum^{b}_{s = a + 2}|q(s)| \geq \frac{(\alpha - 1)}{(b - a - 1)}.
\end{equation}
\end{thm}

\begin{proof}
Let $\mathfrak{B}$ be the Banach space of functions endowed with norm $$\|u\| = \max_{t \in \mathbb{N}^{b}_{a + 1}}|u(t)|.$$ It follows from Theorem \ref{LF Theorem 1} that a solution to \eqref{LF FDE 2} satisfies the equation
\begin{equation*}
u(t) = \sum^{b}_{s = a + 2}G(t, s)q(s)u(s).
\end{equation*}
Hence, 
\begin{align*}
\|u\| & = \max_{t \in \mathbb{N}^{b}_{a + 1}}\Big{|}\sum^{b}_{s = a + 2}G(t, s)q(s)u(s)\Big{|} \\ & \leq \max_{t \in \mathbb{N}^{b}_{a + 1}}\Big{[}\sum^{b}_{s = a + 2}G(t, s)|q(s)||u(s)|\Big{]} \\ & \leq \|u\| \Big{[}\max_{t \in \mathbb{N}^{b}_{a + 1}}\sum^{b}_{s = a + 2}G(t, s)|q(s)|\Big{]} \\ & \leq \|u\| \Big{[}\max_{t \in \mathbb{N}^{b}_{a + 1}, ~ s \in \mathbb{N}^{b}_{a + 2}}G(t, s)\Big{]}\sum^{b}_{s = a + 2}|q(s)|,
\end{align*}
or, equivalently, $$1 \leq \Big{[}\max_{t \in \mathbb{N}^{b}_{a + 1}, ~ s \in \mathbb{N}^{b}_{a + 2}}G(t, s)\Big{]}\sum^{b}_{s = a + 2}|q(s)|.$$ An application of Theorem \ref{LF Green Max} yields the result. 
\end{proof}

Now, we discuss three applications of Theorem \ref{LF Theorem 2}. First, we obtain a criterion for the nonexistence of nontrivial solutions of \eqref{LF FDE 2}.

\begin{thm} \label{LF Theorem 3}
Assume that $1 < \alpha < 2$ and 
\begin{equation}
\sum^{b}_{s = a + 2}|q(s)| < \frac{(\alpha - 1)}{(b - a - 1)}.
\end{equation}
Then, the discrete fractional boundary value problem \eqref{LF FDE 2} has no nontrivial solution on $\mathbb{N}^{b}_{a + 1}$. 
\end{thm}

Next, we estimate a lower bound for eigenvalues of the eigenvalue problem corresponding to \eqref{LF FDE 2}.

\begin{thm} \label{LF Theorem 4}
Assume that $1 < \alpha < 2$ and $u$ is a nontrivial solution of the eigenvalue problem
\begin{equation} \label{LF FDE 3}
\begin{cases}
\big{(}\nabla^{\alpha}_{a}u\big{)}(t) + \lambda u(t) = 0, \quad t \in \mathbb{N}^{b}_{a + 2},\\
\big{(}\nabla^{\alpha - 1}_{a}u\big{)}(a + 1) = 0, ~ u(b) = 0,
\end{cases}
\end{equation}
where $u(t) \neq 0$ for each $t \in \mathbb{N}^{b - 1}_{a + 2}$. Then,
\begin{equation} \label{LF Bound}
|\lambda| \geq \frac{(\alpha - 1)}{(b - a - 1)^{2}}.
\end{equation}
\end{thm}

Finally, we deduce a criterion for the nonexistence of real zeros of certain nabla Mittag-Leffler functions.

\begin{thm} \label{LF Theorem 5}
Let $1 < \alpha < 2$. Then, the function $\lambda E_{-\lambda, \alpha, \alpha - 1}(t, 0) + E_{-\lambda, \alpha, \alpha - 2}(t, 0)$ has no real zeros for $$|\lambda| < \frac{(\alpha - 1)}{(n - 1)^{2}}.$$
\end{thm}

\begin{proof}
Let $a = 0$, $b = n \in \mathbb{N}_{2}$ and consider the eigenvalue problem
\begin{equation} \label{LF FDE 4}
\begin{cases}
\big{(}\nabla^{\alpha}_{0}u\big{)}(t) + \lambda u(t) = 0, \quad t \in \mathbb{N}^{n}_{2},\\
\big{(}\nabla^{\alpha - 1}_{0}u\big{)}(1) = 0, ~ u(n) = 0.
\end{cases}
\end{equation}
By Theorem \ref{SUM ML}, a general solution of \eqref{LF FDE 4} is given by
\begin{equation} \label{S 1}
u(t) = C_{1}E_{-\lambda, \alpha, \alpha - 1}(t, 0) + C_{2}E_{-\lambda, \alpha, \alpha - 2}(t, 0), \quad t \in \mathbb{N}_{1},
\end{equation}
where $C_{1}$, $C_{2} \in \mathbb{R}$. Applying $\nabla^{\alpha - 1}_{0}$ on both sides of \eqref{S 1}, we get
\begin{equation} \label{S 2}
\big{(}\nabla^{\alpha - 1}_{0}u\big{)}(t) = C_{1}E_{-\lambda, \alpha, 0}(t, 0) - \lambda C_{2}E_{-\lambda, \alpha, \alpha - 1}(t, 0), \quad n \in \mathbb{N}_{1}.
\end{equation}
Using $\big{(}\nabla^{\alpha - 1}_{0}u\big{)}(1) = 0$, we get $C_{1} = \lambda C_{2}$. Using $u(n) = 0$, we have that the eigenvalues $\lambda \in \mathbb{R}$ of \eqref{LF FDE 4} are the solutions of
\begin{equation}
\lambda E_{-\lambda, \alpha, \alpha - 1}(n, 0) + E_{-\lambda, \alpha, \alpha - 2}(n, 0) = 0,
\end{equation}
and the corresponding eigenfunctions are given by
\begin{equation}
u(t) = \lambda E_{-\lambda, \alpha, \alpha - 1}(t, 0) + E_{-\lambda, \alpha, \alpha - 2}(t, 0), \quad t \in \mathbb{N}_{1}.
\end{equation}
By Theorem \ref{LF Theorem 2}, if a real eigenvalue $\lambda$ of \eqref{LF FDE 4} exists, i.e. $\lambda$ is a zero of \eqref{LF FDE 4}, then $|\lambda| \geq \frac{(\alpha - 1)}{(n - 1)^{2}}$.

\end{proof}

\section{Right-Focal Type Boundary Value Problem}

In this section, we derive a few properties of the Green's function for a right-focal type discrete boundary value problem and obtain the corresponding Lyapunov-type inequality.

\begin{thm} \label{RF Theorem 1}
Let $1 < \alpha < 2$ and $h : \mathbb{N}^{b}_{a + 1} \rightarrow \mathbb{R}$. The discrete fractional boundary value problem
\begin{equation} \label{RF FDE 1}
\begin{cases}
\big{(}\nabla^{\alpha}_{a}u\big{)}(t) + h (t) = 0, \quad t \in \mathbb{N}^{b}_{a + 2}, \\ u(a + 1) = 0, ~ \big{(}\nabla^{\alpha - 1}_{a}u\big{)}(b) = 0,
\end{cases}
\end{equation}
has the unique solution 
\begin{equation} \label{RF Solution}
u(t) = \sum^{b}_{s = a + 2}G_{r}(t, s)h(s), \quad t \in \mathbb{N}^{b}_{a + 1},
\end{equation}
where
\begin{equation} \label{RF Green}
G_{r}(t, s) = \begin{cases}
\frac{1}{\Gamma(\alpha)}(t - a - 1)^{\overline{\alpha - 1}}, \hspace{108pt} t \in \mathbb{N}^{\rho(s)}_{a + 1},\\
\frac{1}{\Gamma(\alpha)}\big{[}(t - a - 1)^{\overline{\alpha - 1}} - (t - s + 1)^{\overline{\alpha - 1}}\big{]}, \quad t \in \mathbb{N}^{b}_{s}.
\end{cases}
\end{equation}
\end{thm}

\begin{proof}
Using $\big{(}\nabla^{\alpha - 1}_{a}u\big{)}(b) = 0$ in \eqref{Sol 2}, we get $$C_{1} = \frac{1}{\Gamma(\alpha)}\sum^{b}_{s = a + 1}h(s).$$ Using $u(a + 1) = 0$ in \eqref{Sol 1}, we get $$C_{2} = -\frac{1}{\Gamma(\alpha - 1)}\sum^{b}_{s = a + 2}h(s).$$ Substituting the values of $C_{1}$ and $C_{2}$ in \eqref{Sol 1}, we have
\begin{align*}
u(t) & = -\frac{1}{\Gamma(\alpha)}\sum^{t}_{s = a + 1}(t - s + 1)^{\overline{\alpha - 1}}h(s) + \frac{(t - a)^{\overline{\alpha - 1}}}{\Gamma(\alpha)}\sum^{b}_{s = a + 1}h(s) \\ & - \frac{(t - a)^{\overline{\alpha - 2}}}{\Gamma(\alpha - 1)}\sum^{b}_{s = a + 2}h(s) \\ & = -\frac{1}{\Gamma(\alpha)}\sum^{t}_{s = a + 2}(t - s + 1)^{\overline{\alpha - 1}}h(s) + \frac{(t - a)^{\overline{\alpha - 1}}}{\Gamma(\alpha)}\sum^{b}_{s = a + 2}h(s) \\ & - \frac{(t - a)^{\overline{\alpha - 2}}}{\Gamma(\alpha - 1)}\sum^{b}_{s = a + 2}h(s) \\ & = \frac{1}{\Gamma(\alpha)}\sum^{t}_{s = a + 2}\Big{[}(t - a)^{\overline{\alpha - 1}} - (\alpha - 1)(t - a)^{\overline{\alpha - 2}} - (t - s + 1)^{\overline{\alpha - 1}}\Big{]}h(s) \\ & + \frac{1}{\Gamma(\alpha)}\sum^{b}_{s = t + 1}\Big{[}(t - a)^{\overline{\alpha - 1}} - (\alpha - 1)(t - a)^{\overline{\alpha - 2}}\Big{]}h(s) \\ & = \frac{1}{\Gamma(\alpha)}\sum^{t}_{s = a + 2}\Big{[}(t - a - 1)^{\overline{\alpha - 1}} - (t - s + 1)^{\overline{\alpha - 1}}\Big{]}h(s) \\ & + \frac{1}{\Gamma(\alpha)}\sum^{b}_{s = t + 1}(t - a - 1)^{\overline{\alpha - 1}}h(s) \\ & = \sum^{b}_{s = a + 2}G_{r}(t, s)h(s).
\end{align*}
\end{proof}

First, we show that this Green's function is nonnegative and obtain an upper bound for the Green's function and its integral.

\begin{thm} \label{RF Green Positive}
The Green's function $G_{r}(t, s)$ satisfies $G_{r}(t, s) \geq 0$ for $(t, s) \in \mathbb{N}^{b}_{a + 1} \times \mathbb{N}^{b}_{a + 2}$.
\end{thm}

\begin{proof}
For $t \in \mathbb{N}^{\rho(s)}_{a + 1}$, $$G_{r}(t, s) = \frac{(t - a - 1)^{\overline{\alpha - 1}}}{\Gamma(\alpha)} = \frac{\Gamma(t - a + \alpha - 2)}{\Gamma(\alpha)\Gamma(t - a - 1)} \geq 0.$$ Suppose $t \in \mathbb{N}^{b}_{s}$. Since $a + 2 \leq s$, we have $$(t - a - 1)^{\overline{\alpha - 1}} \geq (t - s + 1)^{\overline{\alpha - 1}},$$ implies $$G_{r}(t, s) = \frac{1}{\Gamma(\alpha)}\Big{[}(t - a - 1)^{\overline{\alpha - 1}} - (t - s + 1)^{\overline{\alpha - 1}}\Big{]}\geq 0.$$ Hence the proof.
\end{proof}

\begin{thm} \label{RF Green Max}
The maximum of the Green's function $G_{r}(t, s)$ defined in \eqref{RF Green} is given by $$\max_{(t, s) \in \mathbb{N}^{b}_{a + 1} \times \mathbb{N}^{b}_{a + 2}}G_{r}(t, s) = \frac{(b - a - 1)^{\overline{\alpha - 1}}}{\Gamma(\alpha)}.$$
\end{thm}

\begin{proof}
Clearly, $G_{r}(a + 1, s) = 0$ for each $s \in \mathbb{N}^{b}_{a + 2}$. Fix $t \in \mathbb{N}^{b}_{a + 2}$. For $s \in \mathbb{N}^{b}_{t + 1}$, $\nabla_{s}\big{[}G_{r}(t, s)\big{]} = 0$ implies $G_{r}(t, s)$ is a constant function of $s$. Now, suppose $s \in \mathbb{N}_{a + 1}^{t}$. Consider $$ \nabla_{s}G_{r}(t, s) = \frac{(\alpha - 1)}{\Gamma(\alpha)}(t - s + 2)^{\overline{\alpha - 2}} = \frac{\Gamma(t - s + \alpha)}{\Gamma(\alpha - 1)\Gamma(t - s + 2)} > 0,$$ implies $G_{r}(t, s)$ is an increasing function of $s$. We examine the Green's function to determine whether the maximum for a fixed $t$ will occur at $(t, t)$ or $(t, t + 1)$. We have
\begin{equation}
G_{r}(t, t + 1) = \frac{(t - a - 1)^{\overline{\alpha - 1}}}{\Gamma(\alpha)},
\end{equation}
and
\begin{equation}
G_{r}(t, t) = \frac{(t - a - 1)^{\overline{\alpha - 1}}}{\Gamma(\alpha)} - 1.
\end{equation}
Clearly, $$G_{r}(t, t) < G_{r}(t, t + 1), \quad t \in \mathbb{N}^{b}_{a + 2},$$ and 
\begin{equation*}
\max_{t \in \mathbb{N}^{b}_{a + 1}}G_{r}(t, t + 1) = \frac{(b - a - 1)^{\overline{\alpha - 1}}}{\Gamma(\alpha)}.
\end{equation*}
Thus,
\begin{equation*}
\max_{(t, s) \in \mathbb{N}^{b}_{a + 1} \times \mathbb{N}^{b}_{a + 2}}G_{r}(t, s) = \frac{(b - a - 1)^{\overline{\alpha - 1}}}{\Gamma(\alpha)}.
\end{equation*}
\end{proof}

\begin{thm} \label{FR Green Int Max}
The following inequality holds for the Green's function $G_{r}(t, s)$ from \eqref{RF Green}. $$\sum^{b}_{s = a + 2}G_{r}(t, s) \leq \frac{(b - a - 1)^{\overline{\alpha - 1}}}{\Gamma(\alpha)}(b - a - 1), \quad (t, s) \in \mathbb{N}^{b}_{a + 1} \times \mathbb{N}^{b}_{a + 2}.$$
\end{thm}

\begin{proof}
Consider 
\begin{align*}
\sum^{b}_{s = a + 2}G_{r}(t, s) & = \sum^{t}_{s = a + 2}G_{r}(t, s) + \sum^{b}_{s = t + 1}G_{r}(t, s) \\ & = \frac{1}{\Gamma(\alpha)}\sum^{t}_{s = a + 2}\Big{[}(t - a - 1)^{\overline{\alpha - 1}} - (t - s + 1)^{\overline{\alpha - 1}}\Big{]} \\ & + \frac{1}{\Gamma(\alpha)}\sum^{b}_{s = t + 1}(t - a - 1)^{\overline{\alpha - 1}} \\ & = \frac{(t - a - 1)^{\overline{\alpha - 1}}}{\Gamma(\alpha)}(t  - a - 1) - \sum^{t}_{s = a + 2}\frac{(t - s + 1)^{\overline{\alpha - 1}}}{\Gamma(\alpha)} \\ & = \frac{(t - a - 1)^{\overline{\alpha - 1}}}{\Gamma(\alpha)}(t - a - 1) - \frac{(t - a - 1)^{\overline{\alpha}}}{\Gamma(\alpha + 1)}.
\end{align*}
We now find the maximum of this expression with respect to $t \in \mathbb{N}^{b}_{a + 1}$. Since $$\frac{(t - a - 1)^{\overline{\alpha}}}{\Gamma(\alpha + 1)} \geq 0$$ for $t \in \mathbb{N}^{b}_{a + 1}$, 
\begin{align*}
\max_{t \in \mathbb{N}^{b}_{a + 1}}\Big{[}\sum^{b}_{s = a + 2}G_{r}(t, s)\Big{]} & = \frac{(b - a - 1)}{\Gamma(\alpha)}\max_{t \in \mathbb{N}^{b}_{a + 1}}\big{[}(t - a - 1)^{\overline{\alpha - 1}}\big{]} \\ & = \frac{(b - a - 1)^{\overline{\alpha - 1}}}{\Gamma(\alpha)}(b - a - 1).
\end{align*}
\end{proof}

We are now able to formulate a Lyapunov-type inequality for the right focal boundary value problem.

\begin{thm} \label{RF Theorem 2}
If the following discrete fractional boundary value problem
\begin{equation} \label{RF FDE 2}
\begin{cases}
\big{(}\nabla^{\alpha}_{a}u\big{)}(t) + q(t)y(t) = 0, \quad t \in \mathbb{N}^{b}_{a + 2}, \\ u(a + 1) = 0, ~ \big{(}\nabla^{\alpha - 1}_{a}u\big{)}(b) = 0,
\end{cases}
\end{equation}
has a nontrivial solution, then
\begin{equation} \label{RF Lyp}
\sum^{b}_{s = a + 2}|q(s)| \geq \frac{\Gamma(\alpha)}{(b - a - 1)^{\overline{\alpha - 1}}}.
\end{equation}
\end{thm}

Now, we discuss three applications of Theorem \ref{RF Theorem 2}. First, we obtain a criterion for the nonexistence of nontrivial solutions of \eqref{RF FDE 2}.

\begin{thm} \label{RF Theorem 3}
Assume that $1 < \alpha < 2$ and 
\begin{equation} \label{LF Trivial}
\sum^{b}_{s = a + 2}|q(s)| < \frac{\Gamma(\alpha)}{(b - a - 1)^{\overline{\alpha - 1}}}.
\end{equation}
Then, the discrete fractional boundary value problem \eqref{RF FDE 2} has no nontrivial solution on $\mathbb{N}^{b}_{a + 1}$. 
\end{thm}

Next, we estimate a lower bound for eigenvalues of the eigenvalue problem corresponding to \eqref{RF FDE 2}.

\begin{thm} \label{RF Theorem 4}
Assume that $1 < \alpha < 2$ and $u$ is a nontrivial solution of the eigenvalue problem
\begin{equation} \label{RF FDE 3}
\begin{cases}
\big{(}\nabla^{\alpha}_{a}u\big{)}(t) + \lambda y(t) = 0, \quad t \in \mathbb{N}^{b}_{a + 2},\\
u(a + 1) = 0, ~ \big{(}\nabla^{\alpha - 1}_{a}u\big{)}(b) = 0,
\end{cases}
\end{equation}
where $u(t) \neq 0$ for each $t \in \mathbb{N}^{b - 1}_{a + 2}$. Then,
\begin{equation} \label{RF Bound}
|\lambda| \geq \frac{\Gamma(\alpha)}{(b - a - 1)(b - a - 1)^{\overline{\alpha - 1}}}.
\end{equation}
\end{thm}

Finally, we deduce a criterion for the nonexistence of real zeros of certain nabla Mittag-Leffler functions.

\begin{thm}
Let $1 < \alpha < 2$. Then, the function $E_{-\lambda, \alpha, 0}(t, 0) + \lambda E_{-\lambda, \alpha, \alpha - 1}(t, 0)$ has no real zeros for $$|\lambda| < \frac{\Gamma(\alpha)}{(n - 1)(n - 1)^{\overline{\alpha - 1}}}.$$
\end{thm}

\begin{proof}
Let $a = 0$, $b = n \in \mathbb{N}_{2}$ and consider the eigenvalue problem
\begin{equation} \label{RF FDE 4}
\begin{cases}
\big{(}\nabla^{\alpha}_{0}u\big{)}(t) + \lambda u(t) = 0, \quad t \in \mathbb{N}^{n}_{2},\\
u(1) = 0, ~ \big{(}\nabla^{\alpha - 1}_{0}u\big{)}(n) = 0.
\end{cases}
\end{equation}
Using $u(1) = 0$ in \eqref{S 1}, we get $C_{1} = -C_{2}$. Using $\big{(}\nabla^{\alpha - 1}_{0}u\big{)}(n) = 0$ in \eqref{S 2}, we have that the eigenvalues $\lambda \in \mathbb{R}$ of \eqref{RF FDE 4} are the solutions of
\begin{equation}
E_{-\lambda, \alpha, 0}(n, 0) + \lambda E_{-\lambda, \alpha, \alpha - 1}(n, 0) = 0,
\end{equation}
and the corresponding eigenfunctions are given by
\begin{equation}
u(t) = E_{-\lambda, \alpha, \alpha - 1}(t, 0) - E_{-\lambda, \alpha, \alpha - 2}(t, 0), \quad t \in \mathbb{N}_{1}.
\end{equation}
By Theorem \ref{RF Theorem 2}, if a real eigenvalue $\lambda$ of \eqref{RF FDE 4} exists, i.e. $\lambda$ is a zero of \eqref{RF FDE 4}, then $|\lambda| \geq \frac{\Gamma(\alpha)}{(n - 1)(n - 1)^{\overline{\alpha - 1}}}$. Hence the proof.
\end{proof}

\end{document}